\documentclass[12pt,a4paper]{article}
\usepackage{theorem}
\newtheorem{lemma}{Lemma}
\newtheorem{proposition}[lemma]{Proposition}
\newtheorem{theorem}[lemma]{Theorem}
\newtheorem{corollary}[lemma]{Corollary}
\newtheorem{exampl}[lemma]{Example}
\newenvironment{example}{\begin{exampl}\upshape}{\hfill$\Box$\end{exampl}}
\newcommand{\C}{{\bf C}}

\newcommand{\R}{{\bf R}}

\newcommand{\Z}{{\bf Z}}
\newcommand{\rme}{{\rm e}}
\newcommand{\rmd}{{\rm d}}
\newcommand{\cA}{{\cal A}}
\newcommand{\cB}{{\cal B}}
\newcommand{\cC}{{\cal C}}
\newcommand{\cD}{{\cal D}}
\newcommand{\cE}{{\cal E}}
\newcommand{\cF}{{\cal F}}
\newcommand{\cG}{{\cal G}}
\newcommand{\cH}{{\cal H}}
\newcommand{\cJ}{{\cal J}}
\newcommand{\cK}{{\cal K}}
\newcommand{\cL}{{\cal L}}
\newcommand{\cM}{{\cal M}}

\newcommand{\cS}{{\cal S}}

\newcommand{\cU}{{\cal U}}

\newcommand{\sig}{\sigma}
\newcommand{\alp}{\alpha}
\newcommand{\bet}{\beta}

\newcommand{\del}{\delta}
\newcommand{\eps}{\varepsilon}

\newcommand{\lap}{{\Delta}}
\newcommand{\grad}{\nabla}

\newcommand{\Dom}{{\rm Dom}}

\newcommand{\Ran}{{\rm Ran}}

\newcommand{\norm}{\Vert}

\newcommand{\Schrodinger}{Schr\"odinger }

\newcommand{\pr}{\prime}

\newcommand{\half}{\textstyle{\frac{1}{2}}}

\newcommand{\eqref}[1]{(\ref{#1})}

\newenvironment{proof}{\textbf{Proof}}{\hfill$\Box$}
\newenvironment{proofname}[1]{\textbf{#1}}{$\Box$}

\usepackage{amssymb}
\usepackage{hyperref}
\usepackage{graphicx}
\title{C*-algebras associated with\\ some second order differential
  operators }
\author{E B Davies\thanks{Department of Mathematics, King's College
    London} \and V Georgescu\thanks{CNRS and University of
    Cergy-Pontoise}}
\date{}
\begin{document}
\maketitle
\begin{abstract}
We compare two C*-algebras that have been used to study the essential spectrum. This is done by considering a simple second order elliptic differential operator acting in $L^2(\R^N)$, which is affiliated with one or both of the algebras depending on the behaviour of the coefficients.
\end{abstract}



\section{Introduction}\label{intro}

In the last few years several papers have been written showing that
many of the operators arising in quantum theory lie in one of two
C*-algebras, which we call $\cD$ and $\cE$, each of which contains a
wealth of closed two-sided ideals. The ideals are defined by
considering the asymptotic behaviour of the operators concerned in
various directions at infinity. The ideals allow one to subdivide
the essential spectrum of operators in either algebra into various
geometrically specified parts. The key papers in this context
include \cite{EBD2, Georgescu1} and other papers cited there.

\label{cr0}
This paper arose from exchanges between the authors about which of
the two C*-algebras was better suited to studying the spectrum at
infinity of various operators, particularly elliptic differential
operators.  The contents flesh out their eventual conclusion: the
smaller algebra, $\cE$, suffices for a wide range of uniformly
elliptic operators, but the larger algebra, $\cD$, is needed in a
number of more singular applications.

In this paper we compare $\cD$ and $\cE$ under the assumption that
they are contained in $\cL(\cH)$ where $\cH=L^2(\R^N, \rmd^Nx)$. The
algebra $\cD$ is defined to be the set of all $A\in\cL(\cH)$ such
that
\begin{equation}
\lim_{k\to 0} \norm V_kAV_{-k}-A\norm =0\label{Vproperty}
\end{equation}
where $V_kf(x)=\rme^{ik\cdot x}f(x)$ and $k\in\R^N$. Equivalently $V_k=\rme^{ik\cdot Q}$ where $Q$ is the position operator. The algebra
$\cE$ is the set of all $A\in\cD$ satisfying the further conditions
\begin{equation}
\lim_{s\to 0} \norm U_sA-A\norm =0, \hspace{2em}\lim_{s\to 0} \norm
U_sA^\ast-A^\ast\norm =0.\label{Uproperty}
\end{equation}
where $U_sf(x)=f(x+s)$ and $s\in\R^N$. Equivalently
$U_s=\rme^{is\cdot P}$ where $P$ is the momentum operator. One may
also write $\cE=\cC\rtimes G$, which is the C*-algebra associated
with the action of the group $G$ of all space translations $U_s$ on
the algebra $\cC$ of all uniformly continuous bounded functions on
$\R^N$; see \cite{Wi} for details. Note that $\cE$ contains every operator of the form $f(Q)g(P)$ where $f\in \cC$ and $g\in C_0(\R^N)$.

The algebras $\cD$ and $\cE$ play a remarkable role in the
description of the essential spectrum of partial differential
operators, cf. \cite{EBD2,Georgescu1}. For example, consider the
following natural question: under what conditions is the essential
spectrum $\sigma_{\mathrm ess}(H)$ of a self-adjoint operator $H$ on
$L^2(\R^N)$ determined by its `asymptotic operators', obtained
as limits at infinity of its translates? More precisely, we say that
$H_\varkappa$ is an \emph{asymptotic operator} of $H$ if there is a
sequence $c_n\in \R^{N}$ with $|c_n|\to\infty$ such that
$U_{c_n}HU_{c_n}^*$ converges in strong resolvent sense to
$H_\varkappa$. We refer to \cite{Georgescu1,LS} for recent results
on this question and for other references and note the following
simple answer given in \cite{Georgescu1}: if $H$ is affiliated with
$\cE$ then
\begin{equation}
  \sigma_{\mathrm
    ess}(H)=\overline{\bigcup}_\varkappa\sigma(H_\varkappa),
\label{essspecunion}
\end{equation}
where $\overline{\bigcup}$ means the closure of the union.

\label{cr1}
Note that if $H$ is a self-adjoint operator on some Hilbert space
$\cH$ and $\cA$ is a C*-algebra of operators on $\cH$ then $H$ is
said to be \emph{affiliated to} $\cA$ if $(H-z)^{-1}\in\cA$ for some
complex number $z\not\in\sigma(H)$ (then this clearly holds for all
such $z$).

Although one may also use $\cD$ to study the essential spectrum, the
identity (\ref{essspecunion}) need not hold for operators in
$\cD$. For example, if $\phi$ lies in the space $\cC_0$ of
continuous functions vanishing at infinity, then $\phi(Q)\in\cD$ and
the only asymptotic limit of $\phi(Q)$ is $0$, but the essential
spectrum of $\phi(Q)$ is the closure of the range of $\phi$. We shall
see that similar phenomena occur for differential operators; see
Theorem~\ref{main3-add-bis} and Section~\ref{extras}. In this
context it is relevant that if $\phi\in\cC$ then $\phi(Q)\in \cD$
but $\phi(Q)\notin \cE$ unless $\phi=0$. In particular the identity
operator does not lie in $\cE$.

It might be thought that the failure of (\ref{essspecunion}) for
$\cD$ rules out the use of $\cD$ to investigate the essential
spectrum of operators, but this is not the case. The paper
\cite{EBD2} associates a closed two-sided ideal $\cJ_S$ with every
non-empty open subset $S$ in $\R^N$ (or the relevant underlying
space) and then defines $\sig_S(A)$ for every $A\in\cD$ to be the
spectrum of the image of $A$ in the quotient algebra $\cD/\cJ_s$. It
is shown in \cite[Theorem~18]{EBD2} that $\cJ_S$ contains all
compact operators on $L^2(\R^N,\rmd^Nx)$, but the same proof implies
that it contains $\phi(Q)$ for all $\phi\in \cC_0$. Indeed $\cJ_S$
contains all operators in $\cD$ that have compact support in $\R^N$
on the left and the right in a certain natural sense; see
\cite[Lemma~12]{EBD2}. If $A\in\cD$ then $\sig_S(A)$ captures that
part of the essential spectrum of $A$ that is associated with a
`direction' at infinity determined by the set $S$. See also Theorem~\ref{bddideal} and the comments
after it. \label{cr2}

This paper aims to clarify the role that the two algebras play in
connection with certain second order elliptic differential
operators; the methods can be adapted to higher order operators
under suitable assumptions. The simplest results that we obtain are
Theorems 1, 2 and 3 below. We also present some further theorems
that involve variations of the technical assumptions needed to treat
more general differential operators; see particularly
Section~\ref{extras}.

We start by considering the Friedrichs extension of a non-negative
symmetric operator $H_0$ defined on $C_c^\infty(\R^N)$ by
\begin{equation}
(H_0f)(x)=-\grad \cdot (a\grad f)(x)=-a(x)\lap f(x)- \grad a(x)\cdot\grad
f(x).\label{Hoperator}
\end{equation}
More precisely let $a:\R^N\to (0,\infty)$ be a $C^1$ function and
let $Q_0$ denote the quadratic form associated with $H_0$. It is defined on
$C_c^\infty(\R^N)$ by
\[
Q_0(f)=\langle H_0f,f\rangle=\int_{\R^N} a(x)|\grad f(x)|^2 \, \rmd^N x.
\]
The form $Q_0$ is closable in $L^2(\R^N)$ and its closure $Q$ is
associated with a non-negative self-adjoint operator $H$, called the
Friedrichs extension of $H_0$. And $H$ is affiliated with
$\cD$ (or with $\cE$) if $(H+\alp I)^{-1}$ lies in the relevant
algebra for some, or equivalently all, $\alp>0$. \label{cr3}

Under the assumptions above our main theorems are as follows.

\begin{theorem}\label{main1}
If there exists a constant $c>0$ such that if $0<a(x)\leq c$ for all
$x\in \R^N$ then $A=(H+\alp I)^{-1}$ satisfies (\ref{Vproperty}) for all $\alp>0$, so $H$ is affiliated with $\cD$.
\end{theorem}

\begin{theorem}\label{main2}
If there exists a constant $c>0$ such that $c\leq a(x)<\infty$ for all $x\in \R^N$ then $A=(H+\alp I)^{-1}$ satisfies (\ref{Uproperty}) for all $\alp>0$.
\end{theorem}

If the conditions of both theorems are satisfied, it follows that the operator $H$ is affiliated with $\cE$.

\begin{theorem}\label{main3}
  If $\lim_{|x|\to\infty} a(x)=0$ then $A=(H+\alp I)^{-1}$
  satisfies (\ref{Vproperty}) but not (\ref{Uproperty}) for every
  $\alp >0$. Hence the operator $H$ is affiliated with $\cD$ but not
  with $\cE$.
\end{theorem}

\section{Proof of Theorem 1 and a generalization}\label{theorem1}

In this section we start by proving Theorem~\ref{main1}, and then formulate and prove a more general theorem that has potential for being extended to higher order elliptic operators.

\begin{proofname}{Proof of Theorem \ref{main1}}
We start by observing that it is sufficient to prove that
$\rme^{-Ht}\in\cD$ for all $t>0$, because
$(H+\alpha I)^{-1}=\int_0^\infty \rme^{-t(H+\alpha I)} \rmd t$ in norm.

We start by defining a new Riemannian metric on $\R^N$ by $\rmd
s=a(x)^{-1/2}\rmd x$, where $\rmd x$ is the standard Euclidean
metric. The associated Riemannian volume element is
$\mathrm{dvol}(x)=a(x)^{-N/2}\rmd^Nx$ where $\rmd^Nx$ is the
Euclidean volume element. One sees immediately that
\[
Q(f)=\int_{\R^N} |\grad_n f(x)|^2 \sig(x)^2\,
\mathrm{dvol}(x),\hspace{2em} \norm f\norm_2^2=\int_{\R^N} |f(x)|^2
\sig(x)^2\, \mathrm{dvol}(x)
\]
where $\grad_n f $ is the gradient of $f$ with respect to the new
metric and $\sig(x)=a(x)^{N/4}$. If the distance function  with
respect to the new metric is denoted by $d(x,y)$ then the uniform
bound $0<a(x)\leq c$ implies that \begin{equation} d(x,y)\geq
c^{-1/2}|x-y| \mbox{ for all }x,\, y\in\R^N. \label{dbound}
\end{equation}

If we define the distance $d(E,F)$ between two closed subsets $E,\,
F$ of $\R^N$ by
\[
d(E,F)=\inf\{ d(x,y):x\in E, \, y\in F\},
\]
we are able to apply Lemma~1 of \cite{EBD1}. This states that if
$\phi(x)=\rme^{\alp d(x,E)}$ for some $\alp \geq 0$ then
\[
\norm \phi \rme^{-Ht} f \norm \leq \rme^{\alp^2 t}\norm \phi f\norm
\]
for every $f\in L^2(\R^N)$ and $t>0$.

Let $P_E$ denote the operator on $\cH$ obtained by multiplying by
the characteristic function of $E$, and similarly for $F$. We claim
that
\begin{equation}
\norm P_E \rme^{-Ht}P_F\norm \leq
\rme^{-d(E,F)^2/4t}\label{heatbound}
\end{equation}
for all $t>0$. In order to prove this it is sufficient to establish
that
\[
|\langle \rme^{-Ht}f,g\rangle|\leq \rme^{-d(E,F)^2/4t}\norm f \norm
\, \norm g\norm
\]
for all $f=P_Ef\in L^2(\R^N)$ and $g=P_Fg\in L^2(\R^N)$. We have
\begin{eqnarray*}
|\langle \rme^{-Ht}f,g\rangle|&=& |\langle \phi\rme^{-Ht}f,\phi^{-1}g\rangle|\\
&\leq & \rme^{\alp^2 t} \norm \phi f\norm \, \norm \phi^{-1}g\norm\\
&\leq & \rme^{\alp^2 t} \norm f\norm \, \rme^{-\alp d(E,F)}\norm
g\norm.
\end{eqnarray*}
The proof is completed by putting $\alp=d(E,F)/2t$.

Now let $E_n$ be the unit cube in $\R^N$ with centre $n\in\Z^N$ and
vertices $(n_1\pm\half,...,n_N\pm\half)$, and let $P_n$ be the
corresponding projection. It is immediate that $\norm P_m
\rme^{-Ht}P_n\norm \leq 1$ for all $m,\, n\in \Z^N$ and all $t>0$
but (\ref{dbound}) and (\ref{heatbound}) together imply that
\[
\norm P_m \rme^{-Ht}P_n\norm \leq \rme^{-(|m-n|-k)^2/4ct}
\]
where $k$ depends on $N$. The proof is completed by the use of the
following lemma.
\end{proofname}

\begin{lemma}
If $A$ is any bounded operator on $L^2(\R^N)$ satisfying
\[
\norm P_mAP_n\norm \leq \mu(m-n)
\]
for all $m,n\in \Z^N$, where $\sum_{r\in \Z^N} \mu(r)<\infty$, then
$A\in\cD$.
\end{lemma}

\begin{proof}
Given $r\in \Z^N$, we define $B_r=\sum_{n-m=r}P_mAP_n$ and represent
$L^2(\R^N)$ as the orthogonal direct sum of the subspaces
$L^2(E_n)$. It follows directly from the definition of the operator
norm that $\norm B_r\norm \leq \mu(r)$ and hence that $\sum_{r\in
\Z^N} B_r=A$ as a norm convergent series of operators. It remains
only to prove that each $B_r$ satisfies (\ref{Vproperty}) and hence
lies in $\cD$.

One can prove this by considering each term $P_mAP_n$ independently
provided the norm convergence of $V_kP_mAP_nV_{-k}$ to $P_mAP_n$ is
uniform with respect to $m,\, n$ subject to $n-m=r$. This follows
from the representation
\[
V_kP_mAP_nV_{-k}=\rme^{-ik\cdot r}\left(V_k\rme^{-ik\cdot
m}\right)P_mAP_n\left( V_{-k}\rme^{ik\cdot n}\right).\vspace{-4.9ex}
\]
\end{proof}\vspace{2.3ex}

Our second version of Theorem~\ref{main1} depends on a very general
theorem, that may be applied to higher order elliptic operators of
the type considered in \cite{EBD2}.

\label{cr5}
Given a non-negative self-adjoint operator $H$ on a Hilbert space
$\cH$, we shall use the following concepts and notation freely. Let $\cG=\Dom(H^{1/2})$ be its form domain equipped with the graph topology. Identify $\cG\subset\cH\subset\cG^\ast$ in the usual manner and let $L:\cG\to\cG^\ast$ denote the unique continuous linear operator that extends $H$ from $\Dom(H)$ to $\cG$.

\begin{proposition}\label{groupprop} Let
  $\{V_k\}_{k\in\R^N}$ be a strongly continuous unitary group on
  $\cH$ such that $V_k\cG\subset\cG$ for all $k\in\R^N$. Then the
  restrictions $V_k':=V_k|\cG$ define a $C_0$-group of bounded
  operators on the Hilbert space $\cG$.
\end{proposition}

In applications the conclusions of this proposition are often as
easy to verify as the hypothesis, but a proof may be found in
\cite[Proposition 3.2.5]{ABG}.

From now on we assume that the conditions of Proposition~\ref{groupprop} are satisfied. By taking adjoints we see that each $V_k$ extends to a bounded operator $V_k^{\pr\pr}$ on $\cG^\ast$ and that the $V_k^{\pr\pr}$ form a $C_0$-group of bounded operators on the Hilbert space $\cG^\ast$. In what follows we use the same notation $V_k$ for these three groups, which of them is involved being clear from the context.

\begin{theorem}\label{generaltheorem}
Let $L_k:=V_kLV_{-k}\in\cL(\cG,\cG^*)$. If
\[
\lim_{k\to 0} \norm L_k-L\norm =0
\]
in $\cL(\cG,\cG^*)$, then
\[
\lim_{k\to 0}
\norm V_k\varphi(H)V_{-k}-\varphi(H)\norm =0
\]
in $\cL(\cH)$ for every $\varphi\in C_0(\R)$.
\end{theorem}

\begin{proof}
A standard argument involving the Stone-Weierstrass theorem shows that it suffices to consider the case $\varphi(H)=(H+I)^{-1}\equiv R$. It is easy to see that for each $k\in \R^N$ we have
\[
R_k:=V_kRV_{-k}=(L_k+I)^{-1}|_\cH.
\]
Hence
\begin{eqnarray*}
\|R_k-R\| &=&\norm (L_k+I)^{-1}|_\cH-(L+I)^{-1}|_\cH\norm_{\cL(\cH)}\\
&\leq&\|(L_k+I)^{-1}-(L+I)^{-1}\|_{\cL(\cG^*,\cG)} \\
&\leq& \|(L_k+I)^{-1}\|_{\cL(\cG^*,\cG)}
\|L-L_k\|_{\cL(\cG,\cG^*)}
\|(L+I)^{-1}\|_{\cL(\cG^*,\cG)}\\
&\leq& C \|L-L_k\|_{\cL(\cG,\cG^*)}
\end{eqnarray*}
for some constant $C$ that is independent of $k$ subject to $|k|\leq
1$; this uses the fact that the group $V_k$ is of class $C_0$ in
$\cG$ and $\cG^*$.
\end{proof}

In the following theorem and elsewhere $\cM$ denotes the set of
non-negative, real, self-adjoint $N\times N$ matrices. The
closability assumption of the next theorem holds if for every ball
$B\subset \R^N$ there exists a constant $c_B>0$ such that $a(x)\geq
c_B I$ for all $x\in B$; see \cite[Theorem 1.2.6]{HKST}.

From this point onwards we put $P_r=-i\frac{\partial}{\partial x_r}$, abandoning the convention that $P_r$ denotes a projection, unless this is explicitly stated.

\begin{theorem}\label{main1B}
Let $a:\R^N\to\cM$ be a bounded measurable function and suppose
  that the quadratic form
\begin{equation}
Q(f)=\sum_{r,s=1}^N \langle P_r f, a_{r,s}    \label{cr6}
P_s f \rangle \label{matrixcoeffs}
\end{equation}
on $C_c^1(\R^N)$ is positive and closable. If $H$ is the self-adjoint
operator associated to the closure then $A=(H+\alp I)^{-1}$
satisfies (\ref{Vproperty}) for all $\alp >0$.
\end{theorem}

\begin{proof}
The form domain $\cG$ of $H$ is the completion of $C_c^1(\R^N)$ for the norm $(Q(f)+\|f\|^2)^{1/2}$. Formally $L:\cG\to \cG^\ast$ is given by
\[
L=\sum_{r,s=1}^N P_r a_{r,s} P_s
\quad\mbox{ and }\quad
V_k L V_{-k}= \sum_{r,s=1}^N (P+k)_r\, a_{r,s}
(P+k)_s
\]
Therefore $V_k L V_{-k}$ is a quadratic polynomial in $k\in\R^N$.
The only thing one still has to prove in order to apply
Theorem~\ref{generaltheorem} is that $\cG$ is stable under the
multiplication operators $V_k=\rme^{ik\cdot Q}$. This follows
directly from the inequality
\begin{eqnarray*}
Q(V_kf)&=& \int |Pf + kf|^2_{a(x)}\, \rmd x\\
&\leq& 2\int |Pf|^2_{a(x)}\,\rmd x + 2 \int |kf|^2_{a(x)}\,\rmd x\\
&\leq& 2Q(f) + C |k|^2\|f\|^2
\end{eqnarray*}
where $|\cdot|_{a(x)}$ is the norm on $\C^N$ associated to the
quadratic form $a_{r,s}(x)$.  \label{cr7}
\end{proof}

\section{Proof of Theorem 2 and a generalization}
\label{theorem2}

\begin{proofname}{Proof of Theorem \ref{main2}}
  Let $H$ be a non-negative self-adjoint operator acting in
  $L^2(\R^N,\rmd^Nx)$. Assume that $H\geq H_0=c(-\lap)^m$ in the
  sense of quadratic forms for some $c>0$ and some positive integer
  $m$. This implies that
\[
A=(H_0+I)^{1/2}(I+H)^{-1/2}
\]
is bounded by \cite[Section~4.2]{OPS}. Therefore
\begin{eqnarray*}
\norm (U_s-I)(H+I)^{-1}\norm &=& \norm (U_s-I)(H_0+I)^{-1/2}A(H+I)^{-1/2}\norm\\
&\leq & c \norm (U_s-I)(H_0+I)^{-1/2}\norm.\\
&=& c\sup_{\xi\in\R^N}\left| (\rme^{is\cdot \xi}-1)(1+|\xi|^{2m})^{-1/2}\right|
\end{eqnarray*}
by the use of the Fourier transform. The final expression converges
uniformly to zero as $s\to 0$ by an elementary argument.
\end{proofname}

The proof of Theorem~\ref{main2} can be extended to higher order elliptic operators, but it actually holds in much more generality. Let $\cH=L^2(\R^N,\rmd^N x)$ and let $\cU$ denote the class of all continuous functions $f:\R^N\to [1,\infty)$ such that
$\lim_{|k|\to\infty} f(k)=+\infty$. We define $f(P)$ to be the
unbounded positive self-adjoint operator defined in $\cH$ by
\[
(\cF f(P)\psi)(k)=f(k)(\cF\psi)(k)
\]
where $\cF$ is the Fourier transform and $\Dom(f(P))$ is the set of
all $\psi\in \cH$ such that
\[
\int_{\R^N}\left| f(k)(\cF \psi)(k)\right|^2\, \rmd^N k <\infty.
\]

The following theorem is in \cite[Lemma~3.8]{Georgescu1}, but the proof below is adapted from \cite{EBDscattering}, which only treats the case in which $A$ is compact.

\begin{theorem}\label{aux}
The bounded self-adjoint operator $A$ satisfies (\ref{Uproperty}) if and only if there exists $f\in\cU$ such that $\Ran(A)\subseteq \Dom(f(P))$.
\end{theorem}

\begin{proof}
Suppose that such an $f$ exists. The function $g(k)=\{f(k)\}^{-1}$
is a positive continuous function in $C_0(\R^N)$ and
\[
\norm U_sA-A\norm=\norm (U_sg(P)-g(P))f(P)A\norm \leq \norm
U_sg(P)-g(P)\norm \, \norm f(P)A\norm.
\]
This converges to zero in norm as $s\to 0$ because $(\rme^{ik\cdot
s}-1)g(k)$ converges uniformly to $0$ as $s\to 0$.

Conversely suppose that $A$ lies in the set $\cB$ of all operators satisfying (\ref{Uproperty}). If $h$ lies in the Schwartz
space $\cS$ then
\[
h(P)A=\int_{\R^N}\widetilde{h}(x)U_xA\, \rmd^N x
\]
where $\widetilde{h}$ is the inverse Fourier transform of $h$ and the
integrand is norm continuous. Putting $h_t(k)=\rme^{-k^2t}$ where
$t>0$, or equivalently
\[
\widetilde{h}_t(x)=(4\pi t)^{-N/2} \rme^{-|x|^2/4t}
\]
the assumption that $A\in\cB$ implies that
\[
\lim_{t\to 0}\norm h_t(P)A-A\norm =0.
\]

Now let $t_n$ be a sequence such that $0<t_n\leq 1/2^n$  and $\norm
\{I-h_{t_n}(P)\} A\norm \leq 1/2^n$ for all $n\geq 1$. If
\[
f_M(k)=1+\sum_{n=1}^M (1-h_{t_n}(k))
\]
then $f_M$ is a continuous function on $\R^N$ satisfying $1\leq
f_M(k)\leq M+1$ for all $k$ and $\lim_{|k|\to \infty}f_M(k)=M+1$.
Moreover
\[
\norm f_M(P)A\norm \leq 1+\sum_{n=1}^M \norm \{I-h_{t_n}(P)\} A\norm
<2
\]
for all $M$. If $k\in\R^N$ then
\[
0\leq 1-h_{t_n}(k)=1-\rme^{-|k|^2t_n}\leq \frac{|k|^2}{2^n}
\]
so the sequence $f_M$ increases monotonically and locally uniformly
to a continuous limit $f$. This function $f$ satisfies $1\leq
f(k)\leq 1+|k|^2$ for all $k$ and $\lim_{|k|\to\infty}
f(k)=+\infty$. An application of the closed graph theorem finally
establishes that $\norm f(P)A\norm \leq 2$.
\end{proof}

\section{Proof of Theorem 3 and a generalization}\label{theorem3}

Theorem~\ref{main3} is a special case of the following theorem, which can easily be adapted to higher order elliptic differential operators written in divergence form.

\begin{theorem}\label{main3-add-bis}
  Let $a:\R^N\to \cM$ be a bounded measurable function such that the
  quadratic form $Q_0$ defined on $C_c^1(\R^N)$ by
  (\ref{matrixcoeffs}) is closable. If $H$ is the positive
  self-adjoint operator associated to its closure then $H$ is
  affiliated with $\cD$. If there is a sequence of points
  $c_n\in\R^N$ such that $ c_n\to\infty$ and a sequence of real
  numbers $r_n\to\infty$ such that $\sup_{|x-c_n|\leq r_n}|a(x)|\to
  0$, then $H$ is not affiliated with $\cE$.
\end{theorem}

\begin{proof} The first statement of the theorem is contained in Theorem~\ref{main1B}. Let $H_n=U_{c_n}HU_{-c_n}$ where $c_n\in\R^N$ are as stated. The assumptions of the theorem imply that $\lim_{n\to\infty} \langle H_nf,f\rangle=0$ for all $f\in C_c^1(\R^N)$. Equivalently $\lim_{n\to\infty} H_n^{1/2}f=0$ for all $f\in C_c^1(\R^N)$. Then for such $f$ we have
\[
f-(H_n^{1/2}+I)^{-1}f=(H_n^{1/2}+I)^{-1}H_n^{1/2}f
\]
hence $R_n=(H_n^{1/2}+I)^{-1}$ converges strongly to $I$ as
$n\to\infty$.

Denote $R=(H^{1/2}+I)^{-1}$ and suppose that for each $\eps>0$ there
exists $\del>0$ such that $|s|<\del$ implies $ \norm (U_s-I)R\norm
<\eps$. Clearly $(U_s-I)R_n=U_{c_n}(U_s-I)R U_{-c_n}$ so $\norm
(U_s-I)R_n\norm =\norm (U_s-I)R\norm$. Letting $n\to\infty$ in the
formula
\[
\norm (U_s-I)R_nf\norm <\eps \norm f\norm,
\]
valid for all $f\in L^2(\R^N)$, we obtain $\norm (U_s-I)f\norm \leq
\eps \norm f\norm$ under the same conditions on $s$. But $\norm
U_s-I\norm =2$ for all $s\not= 0$; the contradiction implies that
$R$ does not satisfy (\ref{Uproperty}).  Therefore $R\notin \cE$. An
application of the functional calculus finally implies that
$(H+I)^{-1}\notin\cE$, so $H$ is not affiliated with $\cE$.
\end{proof}

The following comments can be used to give an alternative proof of Theorem~\ref{main3-add-bis} and might be of value in other contexts.

\begin{lemma}\label{shade}
Let $\{c_n\}$ be a sequence (or a net) of points in $\R^N$ and $S$ a
bounded operator on $L^2(\R^N)$ such that the weak limit
$\lim_n U_{c_n}SU_{-c_n}=T$ exists. If $S\in\cD$ then $T\in\cD$.
If $S\in\cE$ then $T\in\cE$.
\end{lemma}

\begin{proof}
If $S_x=U_xSU_{-x}$ then
$V_k S_x V_{-k}=U_xV_k S V_{-k}U_{-x}$
hence
$\norm V_k S_x V_{-k} -S_x \norm = \norm V_k S V_{-k} -S \norm$.
Thus if $f,g\in L^2$ are of norm one then by going to the limit
along $x=c_n$ in the inequality
$|\langle f, (V_k S_x V_{-k} -S_x) g \rangle|
\leq\norm V_k S V_{-k} -S \norm$
we obtain
$\norm V_k T V_{-k} - T \norm \leq \norm V_k S V_{-k} -S
\norm$. This clearly implies the first assertion of the lemma. The
supplementary argument needed for the second part is similar: from
the obvious $\norm (U_s-I) S_x \norm = \norm (U_s-I) S \norm $ we
get $\norm (U_s-I) T \norm \leq \norm (U_s-I) S \norm $ hence the
result.
\end{proof}

\begin{corollary}\label{c:shade}
Let $H$ be a self-adjoint operator affiliated with $\cD$ or $\cE$. Assume that there are a self-adjoint operator $\widetilde H$ and a sequence (or a net) of points $c_n\in\R^N$ such that $\lim_n U_{c_n} H U_{-c_n} = \widetilde H$ in the weak resolvent sense. Then $\widetilde H$ is affiliated with $\cD$ or $\cE$ respectively.
\end{corollary}

\section{Some examples}

In this section we study the theorems of this paper in one dimension, which is particularly simple because the Riemannian metric may be evaluated explicitly in that case.

We assume that $H$ acts in $L^2(\R,\rmd x)$ according to the formula
\[
(Hf)(x)=-\frac{\rmd}{\rmd x}\left( a(x)\frac{\rmd f}{\rmd x}\right)
\]
where $a:\R\to(0,\infty)$ is a continuously differentiable function.
More precisely $H$ is taken to be the Friedrichs extension of the
operator defined initially on $C_c^\infty(\R)$, and it is associated
with the closure of the quadratic form
\[
Q(f)=\int_\R a(x)|f^\pr(x)|^2\, \rmd x
\]
defined initially on $C_c^\infty(\R)$. We do not assume that
$a(\cdot)$ has a uniform positive upper or lower bound, so $H$ might
not be affiliated with either $\cD$ or $\cE$.

We now make the change of variable
\[
s(x)=\int_0^x a(u)^{-1/2}\, \rmd u
\]
so that $\alp<s<\bet$ where $-\infty\leq \alp <0<\bet\leq \infty$.
If we put $g(s)=f(x(s))$, then
\begin{eqnarray*}
\int_\R |f(x)|^2\, \rmd x &=& \int_\alp^\bet |g(s)|^2 \sig(s)^2\rmd s\\
Q(f)&=&\int_\alp^\bet |g^\pr(s)|^2 \sig(s)^2\rmd s
\end{eqnarray*}
where $\sig(s)=a(x)^{1/4}$.

The advantage of the representation (\ref{UformV}) below is that it
often enables one to use the well-developed theory of \Schrodinger
operators to determine the essential spectrum of $K$ and hence of
$H$.

\begin{theorem}\label{fr}
The operator $H$ lies in $\widetilde{\cD}$ where this algebra is the
norm closure of the operators that have finite range in the sense of
\cite{EBD2}, as measured by the Euclidean metric associated with the
variable $s$. If $\sig$ is twice continuously differentiable, then
$H$ is unitarily equivalent to the operator $K$ acting in
$L^2((\alp,\bet),\rmd s)$ according to the formula
\begin{equation}
(Kh)(s)=-h^{\pr\pr}(s)+V(s)h(s)\label{UformV}
\end{equation}
and subject to Dirichlet boundary conditions, where
$V(s)=\sig^{\pr\pr}(s)/\sig(s)$.
\end{theorem}

\begin{proof}
The first part of the theorem follows the same line of argument as
the first proof of Theorem~\ref{main1}. The second part is an
application of \cite{HKST}; one has to put $V=X$ in the proof of
Theorem~4.2.1. The unitary operator $U:L^2((\alp,\bet),\sig(s)^2\rmd
s)\to L^2((\alp,\bet),\rmd s)$ is defined by $(Uf)(s)=\sig(s)f(s)$.
\end{proof}

\begin{example}
The case $a(x)=\rme^{-2x}$ is particularly simple, because we may
then vary the above definition slightly by putting $s(x)=\rme^x$,
where $s$ lies in $(0,\infty)$. We have $\sig(s)=s^{-1/2}$. The new
metric is given by
\[
d(x_1,x_2)= \left| \rme^{x_1}-\rme^{x_2} \right|
\]
in the $x$ variable. The fact that $H$ is affiliated with
$\widetilde{\cD}$ may be interpreted as saying its resolvent has finite
range if unit balls are stretched for large negative $x$ and
compressed for large positive $x$.

If we put $h(s)=s^{1/2}f(s)$ then we obtain
\begin{eqnarray*}
\int_0^\infty |f(s)|^2 \sig(s)^2\, \rmd s&=& \int_0^\infty |h(s)|^2\, \rmd s\\
\int_0^\infty |f^\pr(s)|^2 \sig(s)^2\, \rmd s&=& \int_0^\infty
\left(|h^\pr(s)|^2+\frac{3}{4s^2}|h(s)|^2\right) \, \rmd s.
\end{eqnarray*}

In this representation the operator $H$ becomes
\begin{equation}
(Kh)(s)=-h^{\pr\pr}(s) + \frac{3}{4s^2}h(s)\label{Kdef}
\end{equation}
subject to Dirichlet boundary conditions at $0$ and $\infty$. The
positivity of the potential implies that the (non-negative) Green
function of $(K+I)^{-1}$ is pointwise bounded above by the Green
function of
\[
(K_0h)(s)=-h^{\pr\pr}(s)
\]
and similarly for the heat kernels. This provides an independent
check that $H$ is affiliated with $\widetilde{\cD}$. The formula
(\ref{Kdef}) also allows us to conclude that the spectrum and
essential spectrum of $H$ equal $[0,\infty)$.
\end{example}

\begin{example}
A similar exact calculation may be carried out in the Hilbert space
$L^2((0,\infty), \rmd x)$ for $a(x)=x^\alp$, where $0<\alp<2$.
Putting $\bet=1-\alp/2\in (0,1)$, the new variable $s=x^\bet/\bet$
ranges from $0$ to $\infty$. The corresponding metric is
\[d(x_1,x_2)= \bet^{-1}\left| x_1^\bet-x_2^\bet\right|
\]
which is much larger than the Euclidean metric near $0$ and much
smaller for large $x_1,\, x_2$.
\end{example}

\section{A more general context}\label{extras}

It is interesting to note that one can give descriptions of the
algebras $\cD$ and $\cE$ that are independent of the vector space
structure of $\R^N$. This allows one not only to replace $\R^N$ by a
metric space, \label{cr11}
but also to define certain $C^*$-subalgebras of $\cD$
with which non-uniformly elliptic operators are affiliated.

Let $X$ be a metrizable, locally compact, but non-compact space
equipped with a Radon measure $\mu$ whose support is equal to
$X$. This fixes the Hilbert space $L^2(X)$. Assume that $d$ is a
proper metric compatible with the topology on $X$ such that
$\sup_x\mu(B_x(r))<\infty$ holds for any closed ball $B_x(r)$
(proper means that any closed bounded set is compact). One says that
a bounded operator $A$ on $L^2(X)$ has \emph{$d$-finite range} if
there exists $r>0$ such that $P_EAP_F=0$ for all closed sets $E,F$
such that $d(E,F)>r$. If $X$ is a manifold this is equivalent to
assuming that the distribution kernel of $A$ has support in
$\{(x,y):d(x,y)\leq r\}$. We associate to $d$ two C*-algebras of
operators on $L^2(X)$ by the following rules \cite{EBD2,VG}:
$\cD(d)$ is the norm closure of the set of $d$-finite range
operators and $\cE(d)$ is the norm closure of the set of $d$-finite
range operators which have bounded $d$-uniformly continuous integral
kernels. Our next proposition shows that these definitions provide
natural generalizations of the algebras $\cD$ and $\cE$ considered
before.

\begin{proposition}\label{cd-metric}
If $X=\R^N$ and $d_e$ denotes the Euclidean metric on $X$ then $\cD=\cD(d_e)$ and $\cE=\cE(d_e)$.
\end{proposition}

\begin{proof}
The description of $\cD$ in terms of operators of finite range appears in \cite{EBD2}. The identity $\cD=\cD(d_e)$ is a particular case of \cite[Proposition~7.4]{VG}; the proof is particularly simple when $X=\R^N$. The identity $\cE=\cE(d_e)$ follows from \cite[Proposition 6.5]{VG}. Note that the identification $\cE=\cC\rtimes G$ makes this part obvious.
\end{proof}

Let $\cD$ and $\cE$ be the algebras associated with a triple
$(X,\mu,d)$ satisfying the preceding conditions. We say that a
bounded operator $A$ on $L^2(X)$ has bounded support if there exist
$a\in X$ and $r>0$ such that $A=AP_r=P_rA$, where $P_r$ is the
projection
\[
(P_rf)(x)=\left\{ \begin{array}{ll}
f(x)&\mbox{ if $d(x,a)< r$,}\\
0&\mbox{ otherwise.}
\end{array}\right.
\]
The choice of $a$ is irrelevant in this context. Denote the norm
closure of the set of operators with bounded support by $\cB$. It is
easy to show that every operator with bounded support lies in $\cD$,
hence $\cB\subset\cD$.  \label{cr9} It is also clear that if $X$ is
a discrete space then $\cB$ coincides with the set of compact
operators on $L^2(X)$ (note that the metric is proper).  The
following theorem correctly suggests that $\cB$ plays a similar role
in $\cD$ as the set of compact operators $\cK$ plays in $\cE$.

\begin{theorem}\label{bddideal}
The set $\cB$ is a proper, closed, two-sided ideal in $\cD$.
\end{theorem}

\begin{proof}
This is a special case of \cite[Theorem~18]{EBD2}, in which one takes the set $S$ there to be any non-empty bounded open set in $\R^N$.
\end{proof}

In the light of this theorem, one may define the spectrum of $A\in\cD$ at infinity to be the spectrum of $\pi(A)$, where $\pi:\cD\to\cD/\cB$ is the canonical quotient map. One might even call $A$ a Fredholm operator relative to $\cD$ if $\pi(A)$ is invertible in $\cD/\cB$.

If $X=\R^N$ then Theorems~\ref{main1} and \ref{main2} show that if
one wishes to study uniformly elliptic operators then $d_e$ is the
appropriate choice of the metric. However, Theorem~\ref{fr} shows
that for some second order elliptic operators that are not uniformly
elliptic the relevant metric can be expressed in terms of the second
order coefficients.

The class of operators affiliated to $\cD(d)$ is very large. The main interest of the algebras $\cE(d)$ is that the essential
spectrum of the operators affiliated with them can be described in
terms of asymptotic operators, but it is not so easy to check
that interesting operators are affiliated with $\cE(d)$.  For this
reason, we shall describe a class of Riemannian manifolds for
which the Laplacian is affiliated with the corresponding algebra
$\cE(d)$. The following material extends the scope of the study of the Laplacian on three-dimensional hyperbolic space in \cite[Theorem~44]{EBD2}; in that example the volume doubling condition below does not hold.

Let $X$ be a complete but non-compact Riemannian manifold equipped with its canonical Riemannian measure $\mu$ and distance $d$. Denote
$V_x(r)=\mu(B_x(r))$. Let $H$ be the self-adjoint operator
associated to the (positive) Laplacian $\Delta$ and let $h_t(x,y)$
be the kernel of $\rme^{-tH}$. We assume:

\begin{itemize}

\item the measure has the volume doubling property, i.e. there is a
  constant $D$ such that $V_x(2r)\leq DV_x(r)$ for all $x\in X$ and
  $r>0$;

\item the Poincar\'e inequality holds: there is a constant
  $P$ such that
\[
\int_{B_x(r)} | f-f_{B(r)} |^2 d\mu \leq Pr^2
\int_{B_x(2r)} | \nabla f |^2 d\mu
\]
for all $x\in X$ and $r>0$, where
$f_{B(r)}(x)=V_x(r)^{-1}\int_{B_x(r)}f d \mu$;

\item we have $\sup_{x,y} h_t(x,y)<\infty$ for all $t>0$.

\end{itemize}

\begin{theorem}\label{manifold}
Under the above assumptions on $X$ the operator $H$ is affiliated with $\cE(d)$.
\end{theorem}

\begin{proof}
By using the first two conditions and \cite[Theorem 5.4.12]{SC}
we see that there are constants $C,a>0$ such that
\[
h_t(x,y) \leq
C V_x(\sqrt{t})^{-1}\rme^{-a d(x,y)^2/t},
\]
In particular, the third condition is satisfied if $\inf_x
V_x(\sqrt{t})>0$. Moreover, from \cite[Theorem 5.4.8]{SC} we get for
all $x,y,z,t$ such that $d(y,z)\leq \sqrt{t}$
\begin{equation}\label{eq:hold}
|h_t(x,y)-h_t(x,z)| \leq C t^{-\alpha/2} d(y,z)^{\alpha} h_{2t}(x,y).
\end{equation}
Here $C,\alpha$ are some strictly positive constants. By using also
the third condition we introduced above, we see that the integral
kernel $h=h_1$ of $\rme^{-H}$ is a bounded symmetric H\"older
continuous function, namely there is a number $C$ such that:
\begin{equation}\label{eq:holder}
|h(x,y)-h(x,z)| \leq C d(y,z)^{\alpha} \quad \mathrm{ if } d(y,z)\leq 1.
\end{equation}

We need one more simple argument to complete the proof that $\rme^{-H}\in\cE$. Let $t>0$ and let $\theta:\R\to\R$ be a continuous function such that $0\leq\theta\leq1$, $\theta(t)=1$ if $t\leq r$, and $\theta(t)=0$ if $t\geq r+1$. We set $k(x,y)=h(x,y)\theta(d(x,y))$. Clearly $k$ is a $d$-uniformly continuous $d$-finite range kernel and
\begin{eqnarray*}
\int_{X}|h(x,y)-k(x,y)|\mu(dy)
&\leq& \int_{d(x,y)\geq r} h(x,y)\mu(dy)\\
&\leq& \int_{d(x,y)\geq r} C V_x(1)^{-1}\rme^{-ad(x,y)^2}\mu(dy).
\end{eqnarray*}
Denote $t=1/a$ and observe that the doubling property implies
$V_{x}(\sqrt{t})\leq C(t)V_{x}(1)$. Therefore
\[
\int_{X}|h(x,y)-k(x,y)|\mu(dy)\leq \int_{d(x,y)\geq r}CC(t)^{-1}V_{x}(\sqrt{t})^{-1}\rme^{-d(x,y)^2/t}\mu(dy).
\]
Then from
\cite[Lemma 5.2.13]{SC} we obtain
\[
\int_{X}|h(x,y)-k(x,y)|\mu(dy)\leq K \rme^{-ar^{2}/2}
\]
for some constant $K$ independent of $\theta$ and $r$. From the Schur
lemma it follows that $\|\rme^{-H}-Op(k)\|\leq K \rme^{-ar^{2}/2}$, where
$Op(k)$ is the operator on $L^{2}$ with kernel $k$. Since $Op(k)\in\cE(d)$
and $r>0$ is arbitrary, we get $\rme^{-H}\in\cE(d)$.
\end{proof}

\textbf{Acknowledgements} We thank El Maati Ouhabaz for helpful discussions of the material relating to Theorem~\ref{manifold}.

\end{document}